\documentclass[11pt]{amsart}
\usepackage{amsmath,amsthm,amssymb}

\newtheorem{theorem}{Theorem}

\newtheorem{corollary}{Corollary}
\newtheorem{lemma}{Lemma}

\renewcommand{\ge}{\geqslant}

\renewcommand{\le}{\leqslant}

\renewcommand{\phi}{\varphi}

\def\R{{\mathbb R} }

\def\EN{{\mathcal E} }
\def\M{{\mathcal M} }
\def\D{\kappa}

\title[Regularity of nonlocal minimal cones]{
Regularity of nonlocal 
minimal cones \\ in dimension $2$}

\author{Ovidiu Savin and Enrico Valdinoci}

\begin{document}

\begin{abstract}
We show that the only nonlocal $s$-minimal cones in $\R^2$ are the trivial ones for all $s \in (0,1)$. As a consequence we obtain that the singular set of a nonlocal minimal surface has at most $n-3$ Hausdorff dimension.
\end{abstract}

\maketitle

\section{Introduction}

Nonlocal minimal surfaces were introduced in ~\cite{CRS} as boundaries of measurable sets $E$ whose characteristic function $\chi_E$ minimizes  a certain $H^{s/2}$ norm. 
More precisely, for any~$s\in(0,1)$, the nonlocal ~$s$-perimeter functional ${\rm Per}_s(E,\Omega)$ of a measurable set $E$ in an open set $\Omega \subset \R^n$ is defined as the $\Omega$-contribution of $\chi_E$ in $\|\chi_E\|_{H^{s/2}}$, that is
\begin{equation}\label{F}
{\rm Per}_s(E,\Omega):= 
L(E\cap\Omega,\R^n \setminus E)+L(E\setminus\Omega,\Omega\setminus 
E),\end{equation}
where $L(A,B)$ denotes the double integral
 $$ L(A,B):=\int_A\int_B\frac{dx\,dy}{|x-y|^{n+s}}, \quad \quad \mbox{$A$,$B$ measurable sets.}$$

A set ~$E$ is~{\it $s$-minimal}
in~$\Omega$ if~${\rm Per}_s(E,\Omega)$ is finite and
$$ {\rm Per}_s(E,\Omega)\le {\rm Per}_s(F,\Omega)$$
for any measurable set~$F$ for which~$E\setminus\Omega=F\setminus\Omega$.

We say that~$E$ is~$s$-minimal in~$\R^n$ if
it is~$s$-minimal in any ball ~$B_R$ for any~$R>0$.
The boundary of $s$-minimal sets are referred to as {\it nonlocal $s$-minimal surfaces}.

The theory of nonlocal minimal surfaces developed in ~\cite{CRS} is
(at least for some features)
similar to the theory of standard minimal surfaces. In fact as $s\rightarrow 1^-$, the $s$-minimal surfaces converge to the classical
minimal surfaces and the functional in~\eqref{F} (after a 
multiplication
by a factor of the order of~$(1-s)$) Gamma-converges
to the classical perimeter functional (see~\cite{CV, ADM}).

% Given the regularity of
%the classical minimal hypersurfaces in suitably low dimension,
%a natural question is whether or not also the~$s$-minimal sets
%enjoy some kind of regularity.

In ~\cite{CRS} it was shown that nonlocal $s$-minimal surfaces are $C^{1,\alpha}$ outside a singular set of Hausdorff dimension $n-2$. The precise dimension of the singular set is determined by the problem of existence in low dimensions of a nontrivial global $s$-minimal cone (i.e. an $s$-minimal set~$E$
such that~$tE=E$ for any~$t>0$). In the case of classical minimal surfaces Simons theorem states that the only global minimal cones in dimension $n \le 7$  must be half-planes, which implies that the Hausdorff dimension of the singular set of a minimal surface in $\R^n$
is~$n-8$. In ~\cite{CV2}, the authors used these results to show that if ~$s$ is sufficiently
close to~$1$ the same holds for $s$-minimal surfaces i.e. global $s$-minimal cones must be half-planes if $n \le 7$ and the Hausdorff dimension of the singular set is $n-8$.  

Given the nonlocal character of the functional in \eqref{F}, it seems more difficult to analyze global $s$-minimal cones for general values of~$s\in(0,1)$. The purpose of this paper
is to show that there are no nontrivial $s$-minimal cones
in the plane. Our theorem is the following.

\begin{theorem}\label{THM}
If~$E$ is an $s$-minimal cone in~$\R^2$, then~$E$ is a half-plane.
\end{theorem}

{F}rom Theorem~\ref{THM} above and Theorem~9.4 
of~\cite{CRS},
we obtain that~$s$-minimal sets in two-dimensional domains
are locally~$C^{1,\alpha}$.
Also, from Theorem~\ref{THM} and classical blow-up and blow-down
arguments\footnote{For instance,
one can use the proof of Theorem III.8.17
in~\cite{maggi}, where the density estimates, the compactness
arguments and the
monotonicity formulas for classical minimal surfaces are replaced
by the ones in~\cite{CRS}.

Of course, in all the results presented,
we are implicitly ruling out the trivial case
in which either the $s$-minimal set $E$ or 
its complement is empty.}, we obtain 
that
$s$-minimal sets in the plane are half-planes.
We summarize these observations in the following result:

\begin{corollary}\label{COR}
If~$E$ is an $s$-minimal set in~$\Omega\subset\R^2$, and 
$\Omega'\Subset\Omega$, then~$(\partial 
E)\cap\Omega'$ is a $C^{1,\alpha}$-curve.  

If~$E$ is an $s$-minimal set in~$\R^2$, then~$\partial E$ is a straight 
line.
\end{corollary}

In higher dimensions, by combining the result of Theorem~\ref{THM} here 
with the dimensional reduction performed in~\cite{CRS}, we
obtain that
any nonlocal $s$-minimal surface in~$\R^n$
is locally~$C^{1,\alpha}$ outside a singular set
of Hausdorff dimension~$n-3$.

\begin{corollary}\label{COR2}
Let~$\partial E$ be a nonlocal $s$-minimal surface in~$\Omega\subset\R^n$
and let ~$\Sigma_E \subset \partial E \cap \Omega$ denote its singular set.
Then~${\mathcal{H}}^d(\Sigma_E)=0$
for any~$d>n-3$.
\end{corollary}

The idea of the proof of Theorem \ref{THM} is the following. If $E\subset 
\R^2$ is an $s$-minimal cone then we construct a set $\tilde E$ as a translation of $E$ in $B_{R/2}$ which coincides with $E$ outside $B_R$. 
Then the difference between the energies (of the extension) of $\tilde E$ and $E$ tends to $0$ as $R \to \infty$.  
This implies that also the energy of $E \cap \tilde E$ is arbitrarily close to the energy of $E$.   
On the other hand if $E$ is not a half-plane the set $\tilde E \cap E$ can be modified locally to decrease its energy by a fixed small amount and we reach a contradiction.

In the next section we introduce some notation and obtain the perturbative estimates that are needed for the proof of Theorem ~\ref{THM} in Section~\ref{S2}.

\section{Perturbative estimates}\label{S1}

We start by introducing some notation.

{\bf Notation.}

We denote points in~$\R^n$ by lower case letters, such 
as~$x=(x_1,\dots,x_n)\in\R^n$ and
points in~$\R^{n+1}_+:=\R^n\times(0,+\infty)$
by upper case letters, such
as~$X=(x,x_{n+1})=(x_1,\dots,x_{n+1})\in\R^{n+1}_+$.

The open ball in~$\R^{n+1}$ of radius~$R$ and center $0$ is denoted by~$B_R$.
Also we denote by~$B_R^+:=B_R \cap \R^{n+1}_+$ the open half-ball in~$\R^{n+1}$ and by~$S^n_+:= S^n\cap \R^{n+1}_+$ the unit half-sphere.

The fractional parameter~$s\in(0,1)$ will be fixed
throughout this paper; we also set
$$a:=1-s\in(0,1).$$

The standard Euclidean base of~$\R^{n+1}$
is denoted 
by~$\{e_1,\dots,e_{n+1}\}$. Whenever there is no possibility of confusion we identify $\R^n$ with the hyperplane $\R^n \times \{0\} \subset \R^{n+1}$.

The transpose of a square matrix~$A$ will be denoted by~$A^T$,
and the transpose of a row vector~$V$ is the column vector
denoted by~$V^T$. We denote by~$I$ the identity matrix in $\R^{n+1}$.

\

We introduce the functional
\begin{equation}\label{ENR}
\EN_R (u):=\int_{B_R^+} |\nabla u(X)|^2 x_{n+1}^a \,dX.
\end{equation}
which is related to the $s$-minimal sets
by an extension problem, as shown in Section~7 of~\cite{CRS}.
More precisely, given a set~$E\subseteq\R^n$
with locally finite~$s$-perimeter, we can associate to it uniquely
its extension function~$u:\R^{n+1}_+\rightarrow\R$ 
whose trace on~$\R^n\times\{0\}$ is given by~$
\chi_E-\chi_{\R^n\setminus E}$ and which
minimizes the energy functional in~\eqref{ENR} for any~$R>0$.

We recall (see Proposition~7.3
of~\cite{CRS})
that~$E$ is~$s$-minimal in~$\R^n$ if and only if its 
extension~$u$ is minimal for the energy in~\eqref{ENR} under compact 
perturbations whose trace in~$\R^n\times\{0\}$ takes the values~$\pm1$.
More precisely, for any~$R>0$, 
\begin{equation}\label{2bis}
\EN_R(u)\le\EN_R(v)\end{equation}
for any~$v$ that coincides with~$u$ on~$\partial B_R^+\cap \{x_{n+1}>0\}$
and whose trace on~$\R^n\times\{0\}$ is given 
by~$
\chi_F-\chi_{\R^n\setminus F}$ for any measurable set~$F$
which is a compact perturbation of~$E$ in~$B_R$.

Next we estimate the variation of the functional in~\eqref{ENR}
with respect to horizontal domain perturbations.
For this we introduce a standard cutoff
function
$$ \mbox{~$\phi\in C_0^\infty(\R^{n+1})$,
with~$\phi(X)=1$ if~$|X|\le 1/2$ and~$\phi(X)=0$ if~$|X|\ge 3/4$.}$$
Given~$R>0$, we let
\begin{equation} \label{34}
Y :=X+ \phi(X/R) e_1.\end{equation}
Then we have that~$X\mapsto Y=Y(X)$
is a diffeomorphism of~$\R^{n+1}_+$
as long as~$R$ is sufficiently large
(possibly in dependence of~$\phi$).

Given a measurable function~$u:\R^{n+1}_+\rightarrow\R$,
we define
\begin{equation}\label{A0}
u^+_R(Y):=u(X).\end{equation}
Similarly, by switching~$e_1$ with~$-e_1$ (or~$\phi$
with~$-\phi$ in~\eqref{34}), we can define~$u^-_R(Y)$.

In the next lemma we estimate a discrete second variation for the energy $\EN_R(u)$.
\begin{lemma}\label{P}
Suppose that~$u$
is homogeneous of degree zero and~$\EN_R(u)<+\infty$.
Then
\begin{equation}\label{A00}
\big| \EN_R(u^+_R)+\EN_R(u^-_R)-2\EN_R(u)\big|\le C R^{n-3+a},
\end{equation}
for a suitable~$C\ge0$, depending on~$\phi$ and~$u$.
\end{lemma}

\begin{proof}
We start with the following 
observation.
Let us consider the square matrix of order $(n+1)$
$$A:=
\left(
\begin{matrix}
a_1 & \dots & \dots & a_{n+1}\cr
0 & \dots & \dots & 0\cr
 & \ddots & \cr
0 & \dots & \dots & 0
\end{matrix}
\right)
$$
with~$1+a_1\ne0$. Then a direct computation shows that
\begin{equation}\label{A1}
(I+A)^{-1}=I-\frac{1}{1+a_1} A=I-\frac{A}{
\det(I+ A)}.
\end{equation}
Now, we define
$$ \chi_R(X):=
\left\{\begin{matrix}
1 & {\mbox{ if }} R/2\le |X|\le R, \\
0 & {\mbox{ otherwise }}\end{matrix}
\right.$$
and
$$ \M(X):= \frac1R\left(\begin{matrix}
\partial_1\phi(X/R) & \dots &\dots&\partial_{n+1}\phi(X/R) \cr
0 & \dots & \dots& 0\cr
& \ddots & \cr
0 & \dots & \dots& 0
\end{matrix}
\right).$$
Notice that
\begin{equation}\label{A2} \M=O(1/R) \,\chi_R
.\end{equation}
Let now
$$ \D(X):=|\det D_X Y(X)|=\det (I+\M(X))
=1+\frac{\partial_1\phi(X/R)}{R}=
1+\,{\rm tr}\,\M(X).
$$
By~\eqref{A1}, we see that
\begin{equation}\label{A3} 
\big(D_X Y \big)^{-1}=
\big(I+\M\big)^{-1}=I-\frac{\M}{\D}
.\end{equation}
Also, $1/\D= 1+O(1/R)$, therefore, by~\eqref{A2},
\begin{equation}\label{A4}
\frac{\M \,\M^T}{\D}=
O(1/R^2)\chi_R
.\end{equation}
Now, we perform some chain rule differentiation of
the domain perturbation. 
For this, we take~$X$ to be a function of~$Y$; also,
the functions~$u$, $Y$, $\chi_R$, $\M$ and~$\D$  will be evaluated 
at~$X$,
while~$u^+_R$ will be evaluated at~$Y$
(e.g., the row vector $\nabla_X u$ is a short notation for~$\nabla_X 
u(X)$, while~$\nabla_Y u^+_R$ stands for~$\nabla_Y u^+_R(Y)$).
We use~\eqref{A0} and~\eqref{A3}
to obtain
$$ \nabla_Y u^+_R=
\nabla_X u\, D_Y X =
\nabla_X u\, \big(D_X Y \big)^{-1}
=\nabla_X u\,
\left(I-\frac{\M}{\D}\right).$$
Also, by changing variables,
$$ dY=|\det D_X Y|\,dX=\D\,dX.$$
Accordingly
\begin{eqnarray*}
\big| \nabla_Y u^+_R\big|^2 y_{n+1}^a\,dY &=&
\nabla_X u
\,\left(I-\frac{\M}{\D}\right)
\,\left(I-\frac{\M}{\D}\right)^T\,
\big( \nabla_X u\big)^T x_{n+1}^a\,\D\,dX
\\ &=&
\nabla_X u
\,\left(\D\, I-\M
-\M^T+\frac{\M\M^T}{
\D}\right)\,
\big( \nabla_X u\big)^T x_{n+1}^a\,\,dX
\\ &=&
\nabla_X u
\,\left(\big( 1+\,{\rm tr}\,\M\big)I-{\M}
-{\M^T}+\frac{\M\M^T}{\D}\right)\,
\big( \nabla_X u\big)^T x_{n+1}^a\,dX.
\end{eqnarray*}
Hence, from~\eqref{A4},
\begin{equation*}\begin{split}
&\big| \nabla_Y u^+_R \big|^2 y_{n+1}^a\,dY\\ &\qquad=
\nabla_X u
\,\Big( \big( 1+\,{\rm tr}\,\M\big)I-{\M}
-{\M^T}+O(1/R^2)\chi_R\Big)\,
\big( \nabla_X u\big)^T x_{n+1}^a\,dX.\end{split}\end{equation*}
The similar term for~$\nabla_Y u^-_R$ may be computed
by switching~$\phi$ to~$-\phi$
(which makes~$\M$ switch to~$-\M$): thus we obtain
\begin{equation*}\begin{split}
&\big| \nabla_Y u^-_R\big|^2 y_{n+1}^a\,dY\\ &\qquad=
\nabla_X u
\,\Big( \big( 1-\,{\rm tr}\,\M\big) I+{\M}
+{\M^T}+O(1/R^2)\chi_R\Big)\,
\big( \nabla_X u\big)^T x_{n+1}^a\,dX.
\end{split}\end{equation*}
By summing up the last two expressions, after simplification
we conclude that
\begin{equation}\label{A9}
\Big( \big| \nabla_Y u^+_R\big|^2 +
\big| \nabla_Y u^-_R\big|^2 \Big) y_{n+1}^a\,dY \,=\,
2 \Big(1+O(1/R^2)\chi_R\Big)\,\big|\nabla_X u\big|^2\,x_{n+1}^a\,dX.
\end{equation}
On the other hand, the function~$g(X):=\big|\nabla_X 
u(X)\big|^2\,x_{n+1}^a$ is homogeneous of degree~$a-2$, hence
\begin{equation*}\begin{split}
& \int_{B_R^+} \chi_R \big|\nabla_X u\big|^2\,x_{n+1}^a\,dX
=\int_{B_R^+\setminus B_{R/2}^+} g\,dX
= \int_{R/2}^R \left[ \int_{S^n_+} g(\vartheta\varrho)\,d\vartheta\right]\,
\varrho^n d\varrho
\\ &\qquad=
\int_{R/2}^R \varrho^{n+a-2}
\left[ \int_{S^n_+} 
g(\vartheta )\,d\vartheta\right]\,d\varrho= C R^{n+a-1},
\end{split}\end{equation*}
for a suitable~$C\ge0$ depending on~$u$. This and~\eqref{A9}
give that
\begin{eqnarray*}
&& \int_{B_R^+} \Big( \big| \nabla_Y u^+_R\big|^2 + 
\big| \nabla_Y u^-_R\big|^2 \Big) y_{n+1}^a\,dY
-2\int_{B_R^+} \big|\nabla_X u\big|^2\,x_{n+1}^a\,dX
\\&&\qquad = O(1/R^2)
\int_{B_R^+} \chi_R\big|\nabla_X u\big|^2\,x_{n+1}^a\,dX
\\&&\qquad = O(1/R^2)\cdot C R^{n+a-1}
,\end{eqnarray*}
which completes the proof of the lemma.
\end{proof}

Lemma  \ref{P} turns out to be particularly
useful when~$n=2$.
In this case~\eqref{A00} yields
\begin{equation}\label{09}
\EN_R(u^+_R)+\EN_R(u^-_R)-2\EN_R(u) \le \frac{C}{R^{s}},
\end{equation}
and the right hand side becomes arbitrarily small for large~$R$.
As a consequence, we also obtain the following corollary.

\begin{corollary}\label{CP}
Suppose that~$E$ is an~$s$-minimal cone in~$\R^2$ and 
that~$u$ is the extension
%\footnote{For further details on the extension, recall the beginning of Section~\ref{S1} here or Section~7 of~\cite{CRS}.}
of~$\chi_E-\chi_{\R^2\setminus E}$.
Then
\begin{equation}\label{92bis}
\EN_R(u^+_R) 
\le \EN_R(u)+\frac{C}{R^{s}} .
\end{equation}
\end{corollary}

\begin{proof}
Since~$E$ is a cone, we know that~$u$
is homogeneous of degree zero (see Corollary~8.2 in~\cite{CRS}):
thus, the 
assumptions of Lemma~\ref{P} are fulfilled
and so~\eqref{09} holds true.

{F}rom the minimality of~$u$ (see~\eqref{2bis}), we infer that
\begin{equation*}
\EN_R(u)\le \EN_R(u^-_R),\end{equation*}
which together with~\eqref{09} gives the desired claim.
\end{proof}

\section{Proof of Theorem~\ref{THM}}\label{S2}

We argue by contradiction, by supposing that~$E\subset\R^2$
is an~$s$-minimal cone different than a half-plane.
By Theorem~10.3 in~\cite{CRS},~$E$
is the disjoint union of a finite number of
closed sectors. Then, up
to a rotation, we may suppose that a sector of~$E$ has angle less
than~$\pi$ and
is bisected by~$e_2$. Thus,
there exist~$M\ge1$ and~$p\in B_M$, on the $e_2$-axis, such that $p$
lies in the interior of $E$,
and~$p+ e_1$ and $p-e_1$ lie in the exterior of $E$.

Let $R>4M$ be sufficiently large. Using the notation of Lemma~\ref{P} we have
\begin{equation}\label{76}\begin{split}
&{\mbox{$u^+_R(Y)=u(Y-e_1)$,
for all $Y\in B_{2M}^+$, and}}\\
&{\mbox{$u^+_R(Y)=u(Y)$ for all $Y\in \R^{3}_+\setminus 
B_{R}^+$,}}
\end{split}\end{equation}
where~$u$ is the extension of~$\chi_E-\chi_{\R^2\setminus E}$. 
We define 
\begin{equation*}
 v_R(X):=\min \{ u(X), \,u^+_R(X)\} \quad  {\mbox{ and }} \quad 
w_R(X):=\max \{ u(X), \,u^+_R(X)\}.\end{equation*}
Denote ~$P:=(p,0)\in\R^3$. We claim that
\begin{equation}\label{not}\begin{split}
&{\mbox{$u^+_R<w_R=u$
in a neighborhood of~$P$, and}}\\
&{\mbox{$u<w_R=u^+_R$
in a neighborhood of~$P+e_1$.}}\end{split}\end{equation}
Indeed, by~\eqref{76}
\begin{equation*}
u^+_R(P)=u(P-e_1)=(\chi_E-\chi_{\R^2\setminus 
E})(p-e_1)
=-1\end{equation*}
while
\begin{equation*}
u(P)=(\chi_E-\chi_{\R^2\setminus E})(p)=1.\end{equation*}
Similarly,~$u^+_R(P+e_1)=u(P)=1$ while~$u(P+e_1)=-1$.
This and the continuity of the functions $u$ and $u_R^+$ at $P$, respectively $P+e_1$, give~\eqref{not}.

We point out that~$\EN_R(u)\le \EN_R(v_R)$,
thanks to~\eqref{76} and the minimality of~$u$.
This and the identity
$$
\EN_R(v_R)+\EN_R(w_R)=
\EN_R(u)+\EN_R(u^+_R)$$
imply that
\begin{equation}\label{56}\begin{split}
\EN_R(w_R)\le \EN_R(u^+_R).
\end{split}\end{equation}
Now we observe that
$w_R$ is not a minimizer for~$\EN_{2M}$
with respect to compact perturbations
in~$B_{2M}^+$.
Indeed,
if~$w_R$ were a minimizer we use $u \le w_R$ and the first fact in~\eqref{not} to conclude $u=w_R$ in~$B_{2M}^+$ from the strong maximum principle. However this contradicts the second
inequality in~\eqref{not}.

Therefore, we can modify~$w_R$ inside a compact set of ~$B_{2M}^+$
and obtain a competitor~$u_*$ such that
\begin{equation*}\label{92}
\EN_{2M}(u_*)+\delta \le \EN_{2M}(w_R),
\end{equation*}
for some~$\delta>0$, independent of~$R$
(since~$w_R$ restricted to~$B_{2M}^+$
is independent of~$R$, by~\eqref{76}).

The inequality above implies
\begin{equation}\label{991}
\EN_R (u_*)+\delta \le \EN_R(w_R),\end{equation}
since~$u_*$ and~$w_R$ coincide outside~$B_{2M}^+$.
Thus, we use~\eqref{92bis}, \eqref{56} and~\eqref{991}
to conclude that
$$ \EN_R (u_*)+\delta\le \EN_R(w_R)\le \EN_R(u^+_R)
\le \EN_R(u)+\frac{C}{R^{s}} .$$
Accordingly, if~$R$ is large enough 
we have that~$\EN_R (u_*) < \EN_R(u)$,
which contradicts the minimality of~$u$.
This completes the proof of Theorem~\ref{THM}.

\section*{Acknowledgments}

It is a pleasure to thank Luigi Ambrosio,
Xavier Cabr{\'e} and Giampiero Palatucci for their interesting
comments on a first draft of this paper.

OS has been supported by
NSF grant 0701037. EV has been
supported by MIUR project
``Nonlinear Elliptic problems in the study of vortices and related
topics'', ERC project  ``$\varepsilon$:
Elliptic Pde's and Symmetry of Interfaces and Layers
for Odd Nonlinearities'' and FIRB project ``A\&B: Analysis and Beyond''.
Part of this work was carried out while EV was visiting Columbia
University.

\vfill 

\vspace{1cm}

{{\sc Ovidiu Savin}

Mathematics Department, Columbia University,

2990 Broadway, New York, NY 10027, USA.

Email: {\tt savin@math.columbia.edu}
}

\vspace{1cm}

{{\sc Enrico Valdinoci}

Dipartimento di Matematica, Universit\`a degli Studi di Milano, 

Via Cesare Saldini 50, 20133 Milano, Italy.

Email: {\tt enrico.valdinoci@unimi.it}
}

\end{document}